\documentclass[11pt]{amsart} 
\usepackage{latexsym,amssymb,amsmath,amsthm,tikz}

\theoremstyle{plain}
\newtheorem{thm}{Theorem}[section]
\newtheorem{cor}[thm]{Corollary}
\newtheorem{lem}[thm]{Lemma}
\newtheorem{conj}[thm]{Conjecture}
\newtheorem{prop}[thm]{Proposition}

\tikzstyle{b}=[circle,draw=black,fill=black,thick,inner sep=1.8pt]
\tikzstyle{w}=[circle,draw=black,fill=white,thick,inner sep=1.8pt]
\tikzstyle{v}=[circle,draw=black,fill=red!50,thick,inner sep=1.8pt]
\tikzstyle{r}=[circle,draw=black,fill=red!50,thick,inner sep=1pt]

\theoremstyle{definition}  
\newtheorem{example}{Example}[section]   

\long\def\skipit#1{} 

\begin{document}

\begin{abstract}
Milgram constructed a 28-vertex cubic graph of genus 4 that disproved Duke's conjecture relating Betti number to minimum genus. We apply Milgram's method to construct to find graphs of higher genus violating Duke's conjecture, which gives a sharper bound on that relationship. These graphs are also counterexamples to a related conjecture of Nordhaus \emph{et al.} on the relationship between minimum and maximum genera of graphs. As a side note, we give a simpler proof of correctness for Milgram's method and we show that Duke's conjecture is true for genus at most 3.
\end{abstract}

\title{On Milgram's construction and the Duke embedding conjectures}
\author{Timothy Sun}
\date{}
\maketitle

\section{Introduction}  
\enlargethispage{-12pt}

All graphs in this paper are finite and undirected, and all surfaces in this paper are orientable. Let $G = (V, E)$ be a graph, possibly with self-loops and parallel edges. A natural question regarding the topological properties of a graph is determining into what surfaces that graph embeds. The Edmonds-Heffter correspondence states that any 2-cell embedding $\Pi$ onto a surface can be specified up to equivalence (that is, an orientation-preserving homeomorphism of pairs) by a cyclic ordering of the edge-ends incident with each vertex. For more background on topological graph theory, refer to Gross and Tucker \cite{gt}. Let $S_k$ denote the orientable surface of genus $k$, where $k$ denotes the number of handles. The \emph{(minimum) genus} $\gamma(G)$ is defined to be the minimum value $k$ such that $G$ embeds in $S_k$. Similarly, we define $\Gamma(G)$ to be the \emph{maximum genus}. The \emph{(dimension-1) Betti number} $\beta(G)$ is defined to be the quantity $$|E| - |V| + c(G)$$ where $c(G)$ is the number of components. 

Let $G_1$ and $G_2$ be two graphs, and consider two vertices $v_1 \in V(G_1)$ and $v_2 \in V(G_2)$. The \emph{bar-amalgamation} of $G_1$ and $G_2$ at $v_1$ and $v_2$ is the disjoint union of $G$ and $H$ with an edge between $v_1$ and $v_2$. It is well-known that the minimum genus, maximum genus, and Betti number are all additive under bar-amalgamations. 

Duke \cite{du} demonstrated that for all integers $k$ between $\gamma(G)$ and $\Gamma(G)$, the graph $G$ embeds onto $S_k$, which motivated the study of maximum genus in its own right, by Nordhaus \emph{et al.} \cite{nsw}. Since the number of faces must be positive, it follows from the Euler equation that the maximum genus cannot exceed $\lfloor \beta(G)/2 \rfloor$. Any graph with a 1- or 2-face embedding on a surface has maximum genus exactly equal to $\lfloor \beta(G)/2 \rfloor$ and is said to be \emph{upper-embeddable}.

Two graphs are said to be \emph{homeomorphic} if one can be obtained from the other by subdivisions of edges and smoothing 2-valent vertices. We use the notation $G' \subseteq G$ if $G'$ is homeomorphic to a subgraph of $G$. In a slight abuse of notation, we often do not distinguish between two homeomorphic graphs, except when we emphasize the 2-valent vertices. Furthermore, if $\Pi$ is an embedding of $G$, then let $\Pi|_{G'}$ denote the \emph{induced embedding} on $G'$, where each rotation in $\Pi|_{G'}$ is defined to be the rotation in $\Pi_G$ induced on the edges of $G'$. Since 2-valent vertices only have one cyclic ordering, the graph $G'$ needs only to be homeomorphic to a subgraph of $G$.

Duke \cite{du} was one of the first to consider the relationship between Betti number and minimum genus. From just the Euler equation, one can show that $\beta(G) \geq 2\gamma(G)$, but this bound is weak, even for $\gamma = 1$. Since every nonplanar graph contains either a $K_{3,3}$ or $K_5$ (which have Betti numbers of 4 and 6, respectively), $\beta(G) \geq 4$ when $\gamma(G) \geq 1$. Duke conjectured that a similar relation holds for higher-genus surfaces:

\begin{conj}[Duke's conjecture \cite{du}]
$\beta(G) \geq 4\gamma(G)$ for all graphs $G$.
\label{duke}
\end{conj}

A relationship between Betti number and minimum genus is one step towards a generalization of Kuratowski's theorem. If Duke's conjecture were true, then the family of graphs formed by bar-amalgamations of copies of $K_{3,3}$ show that the bound is the best possible. A graph $G$ is \emph{irreducible} for a surface $S$ if every proper subgraph of $G$ is embeddable in $S$, but $G$ itself is not. Duke's conjecture implies that the smallest irreducible graph for $S_k$ has Betti number $4(k+1)$.

A stronger conjecture by Nordhaus \emph{et al.} \cite{nsw} requires that any graph for which equality holds be upper-embeddable.

\begin{conj}[Nordhaus et al. \cite{nsw}]
$\Gamma(G) \geq 2\gamma(G)$ for all graphs $G$. 
\label{nsw}
\end{conj}

Both conjectures turn out to be incorrect. One source of counterexamples, due to Ungar (see Milgram and Ungar \cite{mu}), is to use cubic graphs of large girth. Nordhaus \cite{no} shows that the girth needs to be at least 12, and the smallest such graph, the Tutte 12-cage, has 126 vertices. One might ask whether or not there are smaller graphs, particularly those for which the obstructions to embedding are more ``local" than just girth considerations. Milgram \cite{mi} found such a counterexample, which we refer to as $M_4$. We describe $M_4$ in the following section, along with our generalizations $M_5$ and $M_6$.

Milgram's graph disproves Duke's conjecture for all $\gamma \geq 4$, but the bound it provides on the Betti number is weak. Our application of Milgram's construction gives sharper bounds for small genera. In the opposite direction, we show that Duke's conjecture is true for $\gamma < 4$. Our main result proves a new bound on the smallest possible Betti number, denoted $g(k)$ for a graph to have genus $k$.

\begin{thm}
For $k = 1,\dotsc,3$, $g(k)$ is $4$, $8$, and $12$, respectively. For $k = 4,\dotsc, 11$, $g(k)$ is no larger than the values in the following table:

$$\begin{array}{c||c|c|c|c|c|c|c|c|c|c}
k & 4 & 5 & 6 & 7 & 8 & 9 & 10 & 11 \\
\hline
g(k) & 15 & 18 & 21 & 25 & 29 & 33 & 36 & 39
\end{array}$$
\end{thm}
\begin{proof}
The first statement follows from an exhaustive computation that demonstrates the correctness of Duke's conjecture for $\gamma \leq 3$ (Theorem \ref{duke3}). Theorem \ref{thm:T5G6} and its corollaries demonstrate the existence of graphs of genus 4, 5, and 6 with Betti numbers 15, 18, and 21, respectively. Bar-amalgamations of copies of these graphs and $K_{3,3}$ yield the remaining upper bounds.
\end{proof}

\bigskip 
\section{Milgram's construction}  

Milgram's cubic graph $M_4$ has 28 vertices and genus 4, disproving both conjectures. We review his construction and use it to build a graph with 40 vertices and genus 6. The notation we use here is essentially the same as \cite{mi}, with some extra terminology. 

Let $H_3$ be the graph formed by taking $K_{2,3}$ and attaching a pendant vertex to each 2-valent vertex, as in Figure \ref{h3}. Each of the 1-valent vertices is called a \emph{free node}, and the edge incident with a free node is a \emph{free edge}. In Milgram's terminology the importance of $H_3$ is that it is ``non-outside," which is the following property:

\begin{figure}[ht]
\begin{tikzpicture}
\draw (0,1) circle (1);
\node (v1) at (0,0)[w]{};
\node (v2) at (-1,1)[v]{};
\node (v3) at (0,1)[v]{};
\node (v4) at (1,1)[v]{};
\node (v5) at (0,2)[w]{};
\node (w2) at (-1.5, .5)[b]{};
\node (w3) at (-.5, 1.5)[b]{};
\node (w4) at (1.5, 1.5)[b]{};
\draw (v1) to (v3) to (v5);
\draw (v2) to (w2);
\draw (v3) to (w3);
\draw (v4) to (w4);
\end{tikzpicture}
\caption{The graph $H_3$. The free nodes are depicted as solid vertices. }
\label{h3}
\end{figure}
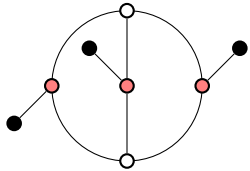

\begin{prop}
An embedding of the graph $H_3$ has a face whose boundary contains all three free nodes if and only if it is a nonplanar embedding.
\label{nonout}
\end{prop}
\begin{proof}
The forward direction is the result of $K_{2,3}$ being non-outerplanar. The reverse direction follows from the fact that $H_3$ has an essentially unique non-planar embedding, which is on a torus, and which has only one face.
\end{proof}

The general idea of Milgram's construction is to attach copies of $H_3$ to a graph, matching each free node of $H_3$ with a vertex of that graph, to produce a new graph of larger genus. Each such attachment of a copy of $H_3$ increases the Betti number of the graph by 4. We present here shortened proofs of Milgram's results.

\begin{prop} \label{prop:scaffold}
Let $G$ be a graph that includes the three 2-valent vertices $u,v,w$, and let $G'$ be the graph obtained by a one-to-one matching of each of the free nodes of the graph $H_3$  to one of those three vertices of $G$.  Let $\Pi$ be an embedding of $G$ in a surface $S$. Then extending $\Pi$ to an embedding of $G'$ can be achieved by adding handles to $S$ as follows:
\begin{enumerate}
\item one handle necessary and sufficient if $u,v,w$ all lie on the boundary of a single face;
\item one handle necessary and sufficient if two of the vertices $u,v,w$ lie on the boundary of one face, but no face boundary contains all three;
\item two handles  necessary and sufficient if no two of the vertices $u,v,w$ lie on the boundary of the same face.
\end{enumerate}   
\end{prop} 

\begin{proof}  \vskip-6pt
Case (1).  By Proposition \ref{nonout}, at least one handle is needed.  To see that one is sufficient, draw $H_3$ in that single face, so that two free nodes are in the exterior, as in Figure \ref{h3}, and match those two free nodes with vertices $u$ and $v$. Then run a handle from the interior region of $H_3$ that contains the other free node to a location near vertex $w$. Finally, re-route the free edge in that interior region to vertex $w$.

Case (2).  Suppose that $u$ and $v$ lie on the boundary of the same face. First match two free nodes of $H_3$ to $u$ and $v$. Next choose a 4-edge path in $H_3$ joining those two free nodes and draw it in that face, and extend this drawing to a drawing of $H_3$.  Then run a handle from whatever face of the drawing contains the other free node and free edge to a location near vertex $w$, and re-route the free edge over that new handle to vertex $w$, as shown in Figure \ref{handle}. 

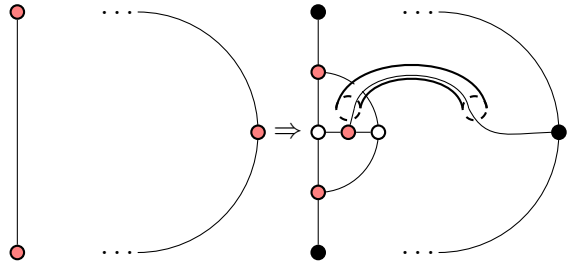
\begin{figure}[ht]
\begin{tikzpicture}[scale=0.8]
\draw (1,-4) arc(-90:90:2);
\node (v1) at (-1, 0)[v]{};
\node (v2) at (-1, -4)[v]{};
\node at (3, -2)[v]{};
\node at (0.7, 0)[]{$\dots$};
\node at (0.7, -4)[]{$\dots$};
\draw (v1) to (v2);
\node at (3.5, -2)[]{$\Rightarrow$};
\begin{scope}[xshift=5cm]
\draw (1,-4) arc(-90:90:2);
\node (v1) at (-1, 0)[b]{};
\node (v2) at (-1, -4)[b]{};
\draw (v1) to (v2);
\draw (-1, -3) arc(-90:43:1);
\draw (-1, -1) arc(90:53:1);
\node (v3) at (-1, -1)[v]{};
\node (v4) at (-1, -3)[v]{};
\node (v5) at (-1, -2)[w]{};
\node (v6) at (0, -2)[w]{};
\node (v7) at (-0.5, -2)[v]{};
\draw (v5) to (v7) to (v6);
\draw (v7) to (-.4, -1.6);
\draw [bend left=80] (-.4, -1.6) to (1.5, -1.6);
\draw (1.5, -1.6) .. controls (1.8, -2.2) and (2.2, -2) .. (3, -2); 
\draw (-.5, -1.6)[dashed, thick] circle (0.2);
\draw (1.6, -1.6)[dashed, thick] circle (0.2);
\draw [bend left=80, thick] (-.7, -1.6) to (1.8, -1.6);
\draw [bend left=80, thick] (-.3, -1.6) to (1.4, -1.6);
\node at (3, -2)[b]{};
\node at (0.7, 0)[]{$\dots$};
\node at (0.7, -4)[]{$\dots$};
\end{scope}
\end{tikzpicture}
\caption{Extending an edge to a drawing of $H_3$ with one extra handle.}
\label{handle}
\end{figure}

Case (3).  With no two of the vertices $u,v,w$ on the same face boudary, it is clear that at least two handles will be needed. That is, in order to draw a path from $u$ to $v$, we need one extra handle, and to connect that path with $w$, another handle is required. The construction of Case (2) is modified here, so that the 4-edge path in $H_3$ joining two free nodes is routed across a first new handle that has one end at a location near vertex $u$ and the other end near vertex $v$.  The rest of $H_3$ is drawn on that handle.  Then a second handle is installed with one end at a location near the remaining free node and the other neat vertex $w$. This case is also illustrated by Figure \ref{handle}.

\end{proof}

Let $G$ be a graph with at least $3n$ 2-valent vertices, and let $$T_n = \{t_0, t_2, \dotsc, t_{n-1}\}$$ be a set of unordered triples of distinct 2-valent vertices of $G$. The \emph{scaffolded graph} $H(G, T_n)$ is the graph formed by taking $n$ copies of $H_3$ and identifying the free nodes of the $i$th copy with the triple $t_i$. 

\begin{example} 
Let $C_n$ be the cycle graph on $n$ vertices. As we might expect, the scaffolded graph $H(C_3, V(C_3))$ has genus $\gamma(C_3) + 1 = 1$, as per Case (1) of Proposition \ref{prop:scaffold}.
\end{example}

\begin{example}
However, consider the wheel graph $W_5$ with three subdivided edges as on the left in Figure \ref{w5}.  Even though $W_5$ is planar, no two of the three subdivision vertices lie on a single face boundary.  As per Case (3) of Proposition \ref{prop:scaffold}, the scaffolded graph on the right has genus 2.
\end{example}

\begin{figure}[ht]
\begin{tikzpicture}
\node (c) at (0,0)[v]{};
\foreach \i in {0,...,4}{
	\node (v\i) at ({cos(72*\i)}, {sin(72*\i)})[v]{};
	\draw (c) to (v\i);
}
\draw (v0) to (v1) to (v2) to (v3) to (v4) to (v0);
\node (a) at (0.5,0)[b, label=below:$b$]{};
\node (b) at ({cos(144)/2},{sin(144)/2})[b, label=below:$a$]{};
\node (c) at (-.25, -.77)[b, label=below:$c$]{};
\end{tikzpicture}
\hspace{15pt}
\begin{tikzpicture}
\begin{scope}[xshift=-3cm]
\draw (0,0) circle (1);
\node (w1) at (0,-1)[v]{};
\node (w2) at (-1,0)[w]{};
\node (w3) at (0,0)[v]{};
\node (w4) at (1,0)[w]{};
\node (w5) at (0,1)[v]{};
\draw (w2) to (w3) to (w4);
\end{scope}
\node (c) at (0,0)[v]{};
\foreach \i in {0,...,4}{
	\node (v\i) at ({cos(72*\i)}, {sin(72*\i)})[v]{};
	\draw (c) to (v\i);
}
\draw (v0) to (v1) to (v2) to (v3) to (v4) to (v0);
\node (a) at (0.5,0)[b]{};
\node (b) at ({cos(144)/2},{sin(144)/2})[b]{};
\node (c) at (-.25, -.77)[b]{};
\draw [bend right=20](w1) to (c);
\draw [bend right=15](w3) to (a);
\draw [bend right=15](w5) to (b);
\end{tikzpicture}
\caption{The wheel $W_5$, and the scaffolded graph $H(W_5, \{a, b, c\})$. The genus of the graph on the right is 2.}
\label{w5}
\end{figure}
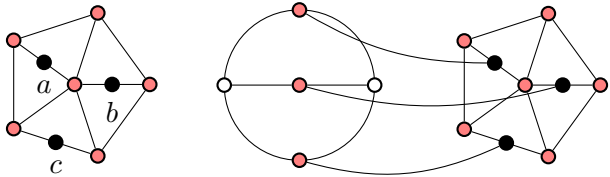

Denote by $S(G, T_n)$ the set of $3^n$ graphs formed by adding an edge between two points of each triple in $T_n$.  We define 
$$\gamma_{T_n}(G) ~= \min_{X \in S(G, T_n)} \gamma(X).$$ 
Milgram gives a general theorem on the genus of a scaffolded graph. Before handling the full generality of Milgram's result, we first consider the case where $n = 1$.

\begin{lem}
$\gamma(H(G, T_1)) = \gamma_{T_1}(G) + 1$.
\label{one}
\end{lem}
\begin{proof}
Suppose first that $\gamma(G) = \gamma_{T_1}(G)$.  Then by Proposition \ref{prop:scaffold}, we have 
$$\gamma(H(G, T_1)) ~\geq~ \gamma(G) + 1 ~=~ \gamma_{T_1}(G) + 1$$
Since an edge can be added to any graph embedding using at most one extra handle, the alternative possibility is that $\gamma_{T_1}(G) = \gamma(G) + 1$. In that case, we replace the edge joining two vertices of the triple $t_0$ by a 4-edge path in $H_3$ joining two free nodes, and we continue as in Case (2) of Proposition \ref{prop:scaffold}.
\end{proof}

The general case is stated similarly and follows from an induction on the number of copies of $H_3$.

\begin{thm}[Milgram \cite{mi}]
$$\gamma(H(G, T_n)) ~=~ \gamma_{T_n}(G) + n.$$
\label{mainthm}
\end{thm}
\begin{proof}
Let $G'$ be a graph in the set $S(G, T_n)$ with the smallest genus. For each added edge $e \in E(G')\setminus E(G)$, we can extend any minimum genus embedding of $G'$ by a drawing of $H_3$ using one extra handle, as described in Figure \ref{handle}, thereby demonstrating that the genus of $H(G, T_n)$ is at most $\gamma(G') + n \,=\, \gamma_{T_n}(G) + n$. 

For the lower bound, consider some minimum genus embedding of the scaffolded graph $H(G, T_n)$. Lemma \ref{one} guarantees that for each triple, we can replace $H_3$ by a single edge between two vertices of the triple and thereby reduce the genus by 1. Replacing each of the $n$ copies of $H_3$ gives us an embedding of a graph $X\in S(G, T_n)$, whose genus is bounded below by $\gamma_{T_n}(G)$. Thus, $\gamma(H(G,T_n)) \geq \gamma(S) + n \geq \gamma_{T_n}(G) + n$.
\end{proof}

Milgram applied his scaffolding construction to a graph we will refer to as $M_4 = H(G_4, T_3)$, where the graph $G_4$ is obtained from the cycle graph $C_9$ by adding two overlapping chords, and where $T_3$ is as shown in Figure \ref{m4}. Milgram demonstrated that $\gamma_{T_3}(G) = 1$, so $\gamma(M_4) = 4$. We prove that same result in the following section as a special case of another genus calculation. Since $M_4$ has 28 vertices and $\beta(M_4) = 15$, it follows that $M_4$ disproves Duke's conjecture.

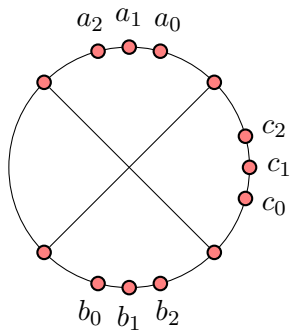
\begin{figure}[ht]
\begin{tikzpicture}[scale=0.8]
\draw (0,0) circle (2);
\foreach \i in {0,...,2}{
  \node (a\i) at ({2*cos(75+15*\i)},{2*sin(75+15*\i)})[v]{};
  \node (b\i) at ({2*cos(-15+15*\i)},{2*sin(-15+15*\i)})[v]{};
  \node (c\i) at ({2*cos(-105+15*\i)},{2*sin(-105+15*\i)})[v]{};
  \node at ({2.5*cos(75+15*\i)},{2.5*sin(75+15*\i)})[]{$a_{\i}$};
  \node at ({2.5*cos(-15+15*\i)},{2.5*sin(-15+15*\i)})[]{$c_{\i}$};
  \node at ({2.5*cos(-105+15*\i)},{2.5*sin(-105+15*\i)})[]{$b_{\i}$};
}
\foreach \i in {0,...,3}
	\node (d\i) at ({2*cos(45+90*\i)}, {2*sin(45+90*\i)})[v]{};
\draw (d0) to (d2);
\draw (d1) to (d3);
\end{tikzpicture}
\caption{A graph whose scaffolded graph is $M_4$. The triples are of the form $t_i = \{a_i, b_i, c_i\}$.}
\label{m4}
\end{figure}

\bigskip
\section{A cubic graph on 40 vertices with genus 6}  

Label the vertices of the cycle graph $C_{15}$ as 
$$a_0, a_1, \dotsc, a_4,\, b_0, \dotsc, b_4,\, c_0, \dotsc, c_4$$ 
in counter-clockwise order. To make our notation consistent with what we previously used, we refer to this graph as $G_6$. Now let $M_6$ be the scaffolded graph $H(G_6, T_5)$, where $T_5$ are the triples of the form $t_i = \{a_i, b_i, c_i\}$, for $i = 0, \dotsc, 4$. To compute the genus of $M_6$, we need to consider the genus of all $3^5$ graphs in $S(G_6, T_5)$. 

\begin{thm}  \label{thm:T5G6}  
$\gamma_{T_5}(G_6) = 1$.
\end{thm}

\begin{proof}
Let $X$ be a graph in  $S(G_6, T_5)$ and consider the added \emph{chords} $E(X)-E(G_6)$.  We call a chord of the form $x_iy_i$ an \emph{$xy$-chord}, where $x,y \in \{a, b, c\}$. To prove that $\gamma_{T_5}(G_6) \le 1$, we observe (perhaps with the aid of a drawing)  that the extension of the cycle graph $G_6$ by the chords 
$$a_0b_0, \quad a_1b_1, \quad b_2c_2, \quad b_3c_3$$ 
is planar.  Thus, the result of adding an additional chord $x_4y_4$ has genus at most 1. Our remaining task is to show that no (allowable) choice of five chords for the graph $X$ is planar.

Since the graph $X$ is cubic, it is perhaps easiest to look for $K_{3,3}$ subgraphs. One visualization of $K_{3,3}$ consists of a circle with three pairs of antipodal points each connected by an edge. We use this picture to prove the nonplanarity of most graphs in $S(G_6,T_5)$. Three chords are said to \emph{overlap} if adding those chords to the cycle $G_6$ forms a $K_{3,3}$.  Note that if $G_6$ is drawn as a circle and the chords are drawn as line segments inside the circle, then the three chords pairwise intersect.

If the graph $X$ has three $ab$-chords, three $bc$-chords, or three $ac$-chords, then those three edges overlap. Otherwise, we may assume that $X$ has two $ab$-chords, two $bc$-chords, and one $ac$-chord. The rotationally-symmetric labeling allows us to make this assumption without loss of generality.

Even without the $ac$-chord, the graph can be nonplanar. Let us momentarily ignore the $ac$-chord and consider only four triples $t_0,\dotsc, t_3$, where the only permissible chords are $ab$- and $bc$-chords. Then the chord $a_3b_3$ overlaps with any two $bc$-chords, and similarly, the chord $b_0c_0$ overlaps with any two $ac$-chords. We are left with two possibilities for chords in the restricted setting:

\begin{enumerate}
\item $a_0b_0$, $a_1b_1$, $b_2c_2$, $b_3c_3$.  
\item $a_0b_0$, $b_1c_1$, $a_2b_2$, $b_3c_3$.
\end{enumerate}

When adding back the fifth triple, there are five possible arrangements depending on the subscript of the $ac$-chord. In Case (1), one can check that any of those choices for the $ac$-chord will overlap with either the two $bc$-chords or the two $ab$-chords as in Figure \ref{case1}. 

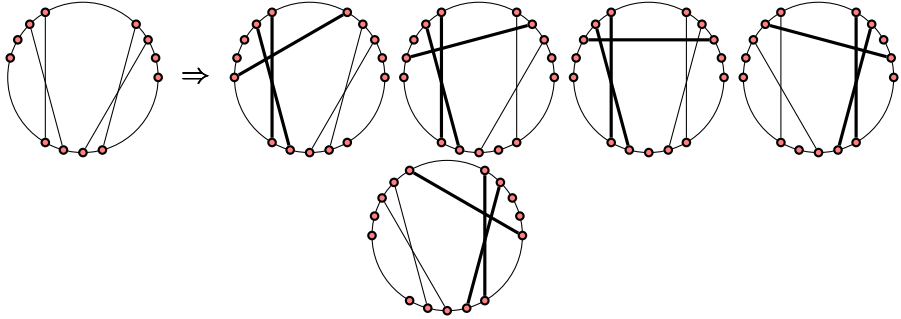
\begin{figure}[ht]
\begin{tikzpicture}
\draw (0,0) circle (1);
\foreach \i in {0,...,3}{
  \node (a\i) at ({cos(120+15*\i)},{sin(120+15*\i)})[r]{};
  \node (b\i) at ({cos(240+15*\i)},{sin(240+15*\i)})[r]{};
  \node (c\i) at ({cos(360+15*\i)},{sin(360+15*\i)})[r]{};
  \node at (1.5, 0)[]{$\Rightarrow$};
}
\draw (a0) to (b0);
\draw (a1) to (b1);
\draw (b2) to (c2);
\draw (b3) to (c3);
\end{tikzpicture}
\begin{tikzpicture}
\draw (0,0) circle (1);
\foreach \i in {0,...,4}{
  \node (a\i) at ({cos(120+15*\i)},{sin(120+15*\i)})[r]{};
  \node (b\i) at ({cos(240+15*\i)},{sin(240+15*\i)})[r]{};
  \node (c\i) at ({cos(360+15*\i)},{sin(360+15*\i)})[r]{};
}
\draw [very thick] (a0) to (b0);
\draw [very thick] (a1) to (b1);
\draw (b2) to (c2);
\draw (b3) to (c3);
\draw [very thick] (c4) to (a4);
\end{tikzpicture}
\begin{tikzpicture}
\draw (0,0) circle (1);
\foreach \i in {0,...,4}{
  \node (a\i) at ({cos(120+15*\i)},{sin(120+15*\i)})[r]{};
  \node (b\i) at ({cos(240+15*\i)},{sin(240+15*\i)})[r]{};
  \node (c\i) at ({cos(360+15*\i)},{sin(360+15*\i)})[r]{};
}
\draw [very thick] (a0) to (b0);
\draw [very thick] (a1) to (b1);
\draw (b2) to (c2);
\draw (b4) to (c4);
\draw [very thick] (c3) to (a3);
\end{tikzpicture}
\begin{tikzpicture}
\draw (0,0) circle (1);
\foreach \i in {0,...,4}{
  \node (a\i) at ({cos(120+15*\i)},{sin(120+15*\i)})[r]{};
  \node (b\i) at ({cos(240+15*\i)},{sin(240+15*\i)})[r]{};
  \node (c\i) at ({cos(360+15*\i)},{sin(360+15*\i)})[r]{};
}
\draw [very thick] (a0) to (b0);
\draw [very thick] (a1) to (b1);
\draw (b3) to (c3);
\draw (b4) to (c4);
\draw [very thick] (c2) to (a2);
\end{tikzpicture}
\begin{tikzpicture}
\draw (0,0) circle (1);
\foreach \i in {0,...,4}{
  \node (a\i) at ({cos(120+15*\i)},{sin(120+15*\i)})[r]{};
  \node (b\i) at ({cos(240+15*\i)},{sin(240+15*\i)})[r]{};
  \node (c\i) at ({cos(360+15*\i)},{sin(360+15*\i)})[r]{};
}
\draw (a0) to (b0);
\draw (a2) to (b2);
\draw [very thick] (b3) to (c3);
\draw [very thick] (b4) to (c4);
\draw [very thick] (c1) to (a1);
\end{tikzpicture}
\begin{tikzpicture}
\draw (0,0) circle (1);
\foreach \i in {0,...,4}{
  \node (a\i) at ({cos(120+15*\i)},{sin(120+15*\i)})[r]{};
  \node (b\i) at ({cos(240+15*\i)},{sin(240+15*\i)})[r]{};
  \node (c\i) at ({cos(360+15*\i)},{sin(360+15*\i)})[r]{};
}
\draw (a1) to (b1);
\draw (a2) to (b2);
\draw [very thick] (b3) to (c3);
\draw [very thick] (b4) to (c4);
\draw [very thick] (c0) to (a0);
\end{tikzpicture}
\caption{Possible $ac$-chords for Case 1. Thickened edges intersect.}
\label{case1}
\end{figure}

For Case (2), this occurs for four of the five possibilities described in Figure \ref{case2}. In the remaining case, where we insert the chord $a_2c_2$, we obtain a graph that has $K_5$ as a minor.
\end{proof}

\begin{figure}[ht]
\begin{tikzpicture}
\draw (0,0) circle (1);
\foreach \i in {0,...,3}{
  \node (a\i) at ({cos(120+15*\i)},{sin(120+15*\i)})[r]{};
  \node (b\i) at ({cos(240+15*\i)},{sin(240+15*\i)})[r]{};
  \node (c\i) at ({cos(360+15*\i)},{sin(360+15*\i)})[r]{};
  \node at (1.5, 0)[]{$\Rightarrow$};
}
\draw (a0) to (b0);
\draw (a2) to (b2);
\draw (b1) to (c1);
\draw (b3) to (c3);
\end{tikzpicture}
\begin{tikzpicture}
\draw (0,0) circle (1);
\foreach \i in {0,...,4}{
  \node (a\i) at ({cos(120+15*\i)},{sin(120+15*\i)})[r]{};
  \node (b\i) at ({cos(240+15*\i)},{sin(240+15*\i)})[r]{};
  \node (c\i) at ({cos(360+15*\i)},{sin(360+15*\i)})[r]{};
}
\draw [very thick] (a0) to (b0);
\draw [very thick] (a2) to (b2);
\draw (b1) to (c1);
\draw (b3) to (c3);
\draw [very thick] (c4) to (a4);
\end{tikzpicture}
\begin{tikzpicture}
\draw (0,0) circle (1);
\foreach \i in {0,...,4}{
  \node (a\i) at ({cos(120+15*\i)},{sin(120+15*\i)})[r]{};
  \node (b\i) at ({cos(240+15*\i)},{sin(240+15*\i)})[r]{};
  \node (c\i) at ({cos(360+15*\i)},{sin(360+15*\i)})[r]{};
}
\draw [very thick] (a0) to (b0);
\draw [very thick] (a2) to (b2);
\draw (b1) to (c1);
\draw (b4) to (c4);
\draw [very thick] (c3) to (a3);
\end{tikzpicture}
\begin{tikzpicture}
\draw (0,0) circle (1);
\foreach \i in {0,...,4}{
  \node (a\i) at ({cos(120+15*\i)},{sin(120+15*\i)})[r]{};
  \node (b\i) at ({cos(240+15*\i)},{sin(240+15*\i)})[r]{};
  \node (c\i) at ({cos(360+15*\i)},{sin(360+15*\i)})[r]{};
}
\draw (a0) to (b0);
\draw (a3) to (b3);
\draw (b1) to (c1);
\draw (b4) to (c4);
\draw (c2) to (a2);
\end{tikzpicture}
\begin{tikzpicture}
\draw (0,0) circle (1);
\foreach \i in {0,...,4}{
  \node (a\i) at ({cos(120+15*\i)},{sin(120+15*\i)})[r]{};
  \node (b\i) at ({cos(240+15*\i)},{sin(240+15*\i)})[r]{};
  \node (c\i) at ({cos(360+15*\i)},{sin(360+15*\i)})[r]{};
}
\draw (a0) to (b0);
\draw (a3) to (b3);
\draw [very thick] (b2) to (c2);
\draw [very thick] (b4) to (c4);
\draw [very thick] (c1) to (a1);
\end{tikzpicture}
\begin{tikzpicture}
\draw (0,0) circle (1);
\foreach \i in {0,...,4}{
  \node (a\i) at ({cos(120+15*\i)},{sin(120+15*\i)})[r]{};
  \node (b\i) at ({cos(240+15*\i)},{sin(240+15*\i)})[r]{};
  \node (c\i) at ({cos(360+15*\i)},{sin(360+15*\i)})[r]{};
}
\draw (a1) to (b1);
\draw (a3) to (b3);
\draw [very thick] (b2) to (c2);
\draw [very thick] (b4) to (c4);
\draw [very thick] (c0) to (a0);
\end{tikzpicture}
\caption{Possible $ac$-chords for Case 2. The middle case does not have three intersecting edges.}
\label{case2}
\end{figure}
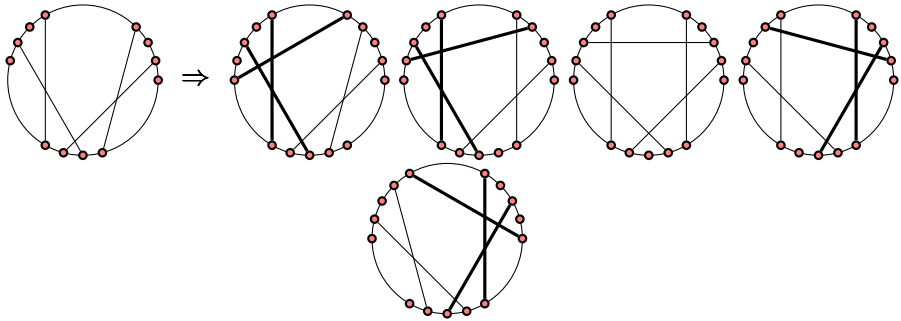

\begin{cor}
$\gamma(M_6) = 6$. 
\end{cor}

We have seen in our proof of Theorem \ref{thm:T5G6} that four triples on a cycle are not enough to force non-planarity. However, suppose we form a new graph $G_5$ by removing the triple $t_0 = \{a_0, b_0, c_0\}$ and adding an edge between $a_0$ and $b_0$. Denote the new set of triples as $T_4$ and define $M_5 = H(G_5, T_4)$. Since every graph in $S(G_5, T_4)$ is also in $S(G_6, T_5)$, it follows that $\gamma_{T_4}(G_5) = 1$ and $\gamma(M_5) = 5$. Moreover, removing the triple $\{a_4, b_4, c_4\}$ from $T_4$ in $G_5$ and adding an edge between $a_4$ and $b_4$ gives us Milgram's graph $M_4$.

\begin{cor}[Milgram \cite{mi}]
$\gamma(M_4) = 4$.
\end{cor}

\bigskip
\section{Maximum genus of scaffolded graphs} 

For any spanning tree $T$ of a graph $G$, we define the \emph{(Xuong) deficiency} $\xi(G, T)$ to be the number of components of $G-T$ with an odd number of edges. The deficiency of a graph $\xi(G)$ is defined to be the minimum deficiency over all spanning trees $T$, and any tree that achieves this minimum is called a \emph{Xuong tree}. Xuong \cite{xu} proved the following relationship between deficiency and maximum genus.

\begin{thm}[Xuong \cite{xu}]
$\Gamma(G) = \frac{1}{2}(\beta(G)-\xi(G)).$
\end{thm}
\smallskip

Using this result, we can show that all the graphs $M_4$, $M_5$, and $M_6$ from the previous section are upper-embeddable. 
\smallskip

\begin{prop}
If $G$ is an upper-embeddable graph, then $H(G, T_n)$ is upper-embeddable for all choices of the set $T_n$ of triples. 
\label{scaffold-upper}
\end{prop}
\begin{proof}
Let $T$ be any Xuong tree of $G$. We extend the tree $T$ into a spanning tree of $H(G, T_n)$ with the thickened edges shown in Figure \ref{xuong} for every copy of $H_3$ in $H(G, T_n)$. Since this operation augments a component of $G-T$ only by an even number of edges, the deficiency remains the same, 0 or 1, which implies that $H(G, T_n)$ is also upper-embeddable. Note that any extension of $T$ that uses all three free edges will work.
\end{proof}

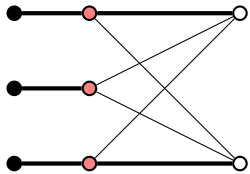
\begin{figure}[ht]
\begin{tikzpicture}
\node (v1) at (3,0)[w]{};
\node (v2) at (1,0)[v]{};
\node (v3) at (1,1)[v]{};
\node (v4) at (1,2)[v]{};
\node (v5) at (3,2)[w]{};
\node (w2) at (0, 0)[b]{};
\node (w3) at (0, 1)[b]{};
\node (w4) at (0, 2)[b]{};
\draw (v1) to (v2) to (v5);
\draw (v1) to (v3) to (v5);
\draw (v1) to (v4) to (v5);
\draw [ultra thick](v1) to (v2);
\draw [ultra thick](v4) to (v5);
\draw [ultra thick](v2) to (w2);
\draw [ultra thick](v3) to (w3);
\draw [ultra thick](v4) to (w4);
\end{tikzpicture}
\caption{Extending a Xuong tree of $G$ to include the attached $H_3$.}
\label{xuong}
\end{figure}

\begin{cor}
The maximum genera of the graphs $M_4$, $M_5$, and $M_6$ are $7$, $9$, and $10$, respectively. 
\end{cor}

With regards to Conjecture \ref{nsw}, all the known counterexamples are either graphs of large girth or scaffolded graphs, and they also violate Duke's conjecture. It seems difficult to construct cubic graphs irreducible for a surface that are not upper-embeddable, and we go as far as to conjecture that no such graph exists. One topic for further exploration is to find a graph that violates Conjecture \ref{nsw} but satisfies Duke's conjecture. Proposition \ref{scaffold-upper} can be thought of as a negative result in this regard.

\bigskip
\section{Duke's conjecture for $\gamma < 4$}  

The iterated bar-amalgamation of $M_4$ with $n$ of copies of $K_{3,3}$ yields a graph whose Betti number is $\beta(M_4)+4n$ and whose minimum genus is $\gamma(M_4)+n$.  Thus, it is a counterexample to Duke's conjecture for all $\gamma \geq 4$. 

In this section, we discuss some work on the remaining cases implying that Duke's conjecture is true for $\gamma < 4$.  It suffices to restrict our search for counterexamples to cubic graphs, since splitting a vertex of degree greater than 3 yields a graph of the same Betti number, without decreasing the minimum genus. Furthermore, deleting vertices of degree 1 and smoothing vertices of degree 2 do not change the minimum genus. 

The Betti number is nonnegative, so for $\gamma = 0$, Duke's conjecture is trivially true.  As mentioned earlier, it is true for $\gamma = 1$, due to Kuratowski's theorem. Milgram (see \cite{mu}) found that all cubic graphs with Betti number 7 can be embedded in the torus, which implies that Duke's conjecture is also true for $\gamma = 2$. By an exhaustive enumeration, Milgram also calculated that no graph of Betti number at most 10 has genus 3.

In order to reduce the amount of computer calculation, we require a well-known fact about irreducible graphs.

\begin{prop}[e.g. Glover \emph{et al.} \cite{ghw}]
A cubic graph which is irreducible for a surface must be triangle-free.
\end{prop}
\begin{proof}
Suppose that $G$ is cubic and irreducible for $S_k$, and let $v_1, v_2, v_3$ be three vertices of some triangle of $G$. Then, $G-v_1v_2$ embeds onto $S_{k-1}$. However, since $v_3$ is of degree 3, the vertices $v_1$ and $v_2$ lie on the same face of that embedding, which is a contradiction.  
\end{proof}

Our own computation verifies the calculations done by Milgram. Using McKay's \texttt{geng} program \cite{mc}, we generated the 97292 simple, triangle-free, 2-connected cubic graphs on 20 vertices and found that there are 151, 33956, and 63185 graphs of genus 0, 1, and 2, respectively, and none of genus 3 or higher. Thus, for any cubic graph of genus 3, the number of vertices is at least 22, and, accordingly, the Betti number is at least 12. One can infer the following result:

\begin{thm}
Duke's conjecture is true for $\gamma \leq 3$.
\label{duke3}
\end{thm}

\section{Conclusion}   

We applied Milgram's scaffolding construction in order to exhibit new counterexamples to Duke's conjecture and provide sharper upper bounds on the minimum possible Betti number for graphs of a given genus. In explaining the construction, we gave a simpler proof of correctness in the form of Proposition \ref{prop:scaffold}. In particular, the extension given in Figure \ref{handle} is  slightly more direct than Milgram's \cite{mi}, and our proof that two handles are necessary in Case (3) is significantly shorter. We also demonstrated that Milgram's construction often produces upper-embeddable graphs, so it does not give us stronger counterexamples for Conjecture \ref{nsw}.

Finally we demonstrated that Duke's conjecture is true for $\gamma \leq 3$, completing an unfinished exhaustive computation of Milgram, as reported in \cite{mu}.

\end{document}